\documentclass{amsart}
\newtheorem{thm}{Theorem}
\newtheorem{lem}[thm]{Lemma}
\newtheorem{cor}[thm]{Corollary}

\theoremstyle{definition}

\newtheorem{ex}[thm]{Example}

\providecommand{\abs}[1]{\lvert#1\rvert}

\usepackage{color}

\makeatletter
\newcommand{\superimpose}[2]{%
  {\ooalign{$#1\@firstoftwo#2$\cr\hfil$#1\@secondoftwo#2$\hfil\cr}}}
\makeatother

\begin{document}
\title[Finitely additive mass transportation]{Finitely additive mass transportation}

\author{Pietro Rigo}
\address{Pietro Rigo, Dipartimento di Scienze Statistiche ``P. Fortunati'', Universit\`a di Bologna, via delle Belle Arti 41, 40126 Bologna, Italy}
\email{pietro.rigo@unibo.it}



\begin{abstract}
Some classical mass transportation problems are investigated in a finitely additive setting. Let $\Omega=\prod_{i=1}^n\Omega_i$ and $\mathcal{A}=\otimes_{i=1}^n\mathcal{A}_i$, where $(\Omega_i,\mathcal{A}_i,\mu_i)$ is a ($\sigma$-additive) probability space for $i=1,\ldots,n$. Let $c:\Omega\rightarrow [0,\infty]$ be an $\mathcal{A}$-measurable cost function. Let $M$ be the collection of finitely additive probabilities on $\mathcal{A}$ with marginals $\mu_1,\ldots,\mu_n$. If couplings are meant as elements of $M$, most classical results of mass transportation theory, including duality and attainability of the Kantorovich inf, are valid without any further assumptions. Special attention is devoted to martingale transport. Let $(\Omega_i,\mathcal{A}_i)=(\mathbb{R},\mathcal{B}(\mathbb{R}))$ for all $i$ and
$$M_1=\bigl\{P\in M:P\ll P^*\text{ and }(\pi_1,\ldots,\pi_n)\text{ is a }P\text{-martingale}\}$$
where $P^*$ is a reference probability on $\mathcal{A}$. If $M_1\ne\emptyset$, then
$$\int c\,dP=\inf_{Q\in M_1}\int c\,dQ\quad\quad\text{for some }P\in M_1.$$
Conditions for $M_1\ne\emptyset$ are given as well.
\end{abstract}

\maketitle

\section{Introduction}\label{intro}

\noindent Mass transportation is nowadays a dynamic field of research. Its applications range in a number of fields, including probability theory, differential equations, geometric measure theory, economics and finance; see e.g. \cite{AGS}, \cite{RRLIBRO}, \cite{VILLA}.

\medskip

\noindent This paper deals with mass transportation problems when couplings are finitely additive probabilities. To be more precise, and to highlight similarities and differences between our approach and the usual one, we need to recall the standard framework where transportation problems are investigated. In the sequel, the abbreviation {\em f.a.p.} stands for {\em finitely additive probability} and a {\em probability measure} is a $\sigma$-{\em additive} f.a.p. Moreover, we use the notation
$$P(f)=\int f\,dP$$
whenever $P$ is a f.a.p. and $f$ a function such that $\int f\,dP$ is well defined.

\subsection{The standard framework}\label{z34d8ng1q}
Let $I=\{1,\ldots,n\}$ where $n$ is a positive integer. For each $i\in I$, let $(\Omega_i,\mathcal{A}_i)$ be a measurable space and $\mu_i$ a probability measure on the $\sigma$-field $\mathcal{A}_i$. Define
\begin{gather*}
\Omega=\Omega_1\times\ldots\times\Omega_n\quad\text{and}\quad\mathcal{A}=\mathcal{A}_1\otimes\ldots\otimes\mathcal{A}_n
\end{gather*}
and denote by $\pi_i:\Omega\rightarrow\Omega_i$ the $i$-th canonical projection, namely,
$$\pi_i(\omega)=\omega_i\quad\quad\text{for all }i\in I\text{ and }\omega=(\omega_1,\ldots,\omega_n)\in\Omega.$$
Moreover, let $c:\Omega\rightarrow [-\infty,\infty]$ be a (cost) function. Various conditions on $c$ can be taken into account. In this paper, $c$ {\em is $\mathcal{A}$-measurable and takes values in $[0,\infty]$. In particular, it may be that $c=\infty$}.

\medskip

\noindent A {\em coupling} (or a {\em transport plan}) is a probability measure $P$ on $\mathcal{A}$ having $\mu_1,\ldots,\mu_n$ as marginals, in the sense that
$$P\circ \pi_i^{-1}=\mu_i\quad\text{for all }i\in I.$$
The collection of all couplings, henceforth denoted by $\Gamma$, plays a basic role. A few classical issues are:

\medskip

\begin{itemize}

\item[(i)] Give conditions for the existence of $P\in\Gamma$ such that
\begin{gather}\label{kant}
P(c)=\inf_{Q\in\Gamma}Q(c);
\end{gather}

\medskip

\item[(ii)] Characterize those $P\in\Gamma$ satisfying equation \eqref{kant} (provided they exist);

\medskip

\item[(iii)] Give conditions for the duality relation
\begin{gather*}
\inf_{Q\in\Gamma}Q(c)=\sup_{f_1,\ldots,f_n}\,\sum_{i=1}^n\mu_i(f_i)
\end{gather*}
where $\sup$ is over the $n$-tuple $(f_1,\ldots,f_n)$ such that $f_i\in L_1(\mu_i)$ for all $i\in I$ and $\sum_{i=1}^nf_i\circ\pi_i\le c$.

\end{itemize}

\medskip

\noindent Here, $L_1(\mu_i)=L_1(\Omega_i,\mathcal{A}_i,\mu_i)$ is the class of $\mathcal{A}_i$-measurable functions $f:\Omega_i\rightarrow\mathbb{R}$ such that $\mu_i(\abs{f})=\int\abs{f}\,d\mu_i<\infty$ (without identifying maps which agree $\mu_i$-a.s.).

\medskip

\noindent A natural development is to fix a subset $\Gamma_0\subset\Gamma$ and to investigate (i)-(ii)-(iii) (and possibly other problems) with $\Gamma_0$ in the place of $\Gamma$. Following \cite{ZAEV}, for instance, $\Gamma_0$ could be
\begin{gather*}
\Gamma_0=\bigl\{P\in\Gamma:P(\abs{f})<\infty\text{ and }P(f)=0\text{ for all }f\in F\bigr\}
\end{gather*}
where $F$ is a given class of $\mathcal{A}$-measurable functions $f:\Omega\rightarrow\mathbb{R}$. As a special case, when $(\Omega_i,\mathcal{A}_i)=(\mathbb{R},\mathcal{B}(\mathbb{R}))$ for each $i\in I$, a suitable choice of $F$ yields
\begin{gather*}
\Gamma_0=\bigl\{P\in\Gamma:(\pi_1,\ldots,\pi_n)\text{ is a martingale under }P\bigr\}.
\end{gather*}
Such a $\Gamma_0$, introduced in \cite{BHLP2013}, corresponds to the so-called {\em martingale transport}. In addition to be theoretically intriguing, martingale transport has solid financial motivations; see e.g. \cite{BHLP2013}, \cite{BNT2017}, \cite{GHLT2014} and references therein.

\medskip

\noindent A last remark is that, for any choice of $\Gamma_0$, a preliminary question is whether $\Gamma_0\ne\emptyset$. In martingale transport, for instance, as a consequence of some results by Strassen \cite{STR}, one obtains $\Gamma_0\ne\emptyset$ if and only if
$$\int\abs{x}\,\mu_i(dx)<\infty\quad\text{and}\quad\mu_i(f)\le\mu_{i+1}(f)$$
for all $i\in I$ and all convex functions $f:\mathbb{R}\rightarrow\mathbb{R}$.

\medskip

\subsection{Content of this paper} Investigating mass transportation in a finitely additive setting is a quite natural idea, and various hints in this direction are scattered throughout the literature; see e.g. \cite{EKSO}, \cite{LUTH}, \cite{RRLIBRO}, \cite{RUS1991} and references therein. Usually, however, f.a.p.'s are only instrumental. Typically, a result concerning f.a.p.'s is nothing but an intermediate step toward the corresponding $\sigma$-additive result. Instead, f.a.p.'s have an intrinsic interest in several mass transportation problems; see e.g. Examples \ref{f90h4qj7}-\ref{ab6ym1l}. Nevertheless, to the best of our knowledge, no systematic treatment of the finitely additive mass transportation is available to date. This paper aims to fill this gap in the special case where $\mu_1,\ldots,\mu_n$ are probability measures.

\medskip

\noindent Let
$$\mathbb{P}=\bigl\{\text{all f.a.p.'s on }\mathcal{A}\bigr\}$$
and
$$M=\bigl\{P\in\mathbb{P}:P\circ \pi_i^{-1}=\mu_i\text{ for all }i\in I\bigr\}.$$
In this paper, problems (i)-(ii)-(iii) are investigated with $M$ in the place of $\Gamma$. Similarly, the subsets $\Gamma_0\subset\Gamma$ are replaced by the corresponding subsets $M_0\subset M$.

\medskip

\noindent Our main result is that, if $\Gamma$ is replaced by $M$, each of problems (i)-(ii)-(iii) admits a solution assuming only that $c$ is non-negative and $\mathcal{A}$-measurable. On the contrary,  to have a solution in the standard framework, further conditions on $c$ and/or $\mu_1,\ldots,\mu_n$ are needed (such conditions are recalled at the end of Subsection \ref{um7m}).

\medskip

\noindent Special attention is devoted to martingale transport. To illustrate, suppose $(\Omega_i,\mathcal{A}_i)=(\mathbb{R},\mathcal{B}(\mathbb{R}))$ for each $i\in I$ and define
$$M_0=\bigl\{P\in M:(\pi_1,\ldots,\pi_n)\text{ is a }P\text{-martingale}\}.$$
If $M_0\ne\emptyset$, then $P(c)=\inf_{Q\in M_0}Q(c)$ for some $P\in M_0$. Similarly, fix a reference f.a.p. $P^*\in\mathbb{P}$ and define
$$M_1=\bigl\{P\in M:P\ll P^*\text{ and }(\pi_1,\ldots,\pi_n)\text{ is a }P\text{-martingale}\}.$$
Once again, if $M_1\ne\emptyset$, there is $P\in M_1$ such that $P(c)=\inf_{Q\in M_1}Q(c)$.

\medskip

\noindent A remark on $M_1$ is in order. Suppose $P^*$ is a probability measure. In the standard martingale transport, to our knowledge, the set
$$M_1\cap\Gamma=\bigl\{P\in\Gamma:P\ll P^*\text{ and }(\pi_1,\ldots,\pi_n)\text{ is a }P\text{-martingale}\}$$
is largely neglected. The only partial exception is \cite{KORMCC}, which is focused on
$$\Gamma_1=\bigl\{P\in\Gamma:P\le u\,P^*\bigr\}\quad\quad\text{where }u>0\text{ is a given constant.}$$
There are two main differences between such a $\Gamma_1$ and $M_1\cap\Gamma$. The elements of $\Gamma_1$ need not be martingale probability measures, and $P\ll P^*$ is strengthened into $P\le u\,P^*$ (that is, not only $P$ has a density with respect to $P^*$, but this density is bounded by a {\em given} constant $u$). However, to investigate $M_1\cap\Gamma$ looks quite natural. In fact, the usual motivation for martingale transport is that martingale probability measures play a role in various financial problems. But, in most such problems, probability measures are also required to be equivalent, or at least absolutely continuous, with respect to some $P^*$. To sum up, to focus on $M_1\cap\Gamma$ is reasonable in the standard martingale transport. In turn, in the framework of this paper, it is natural to focus on $M_1$.

\medskip

\noindent The result quoted above requires $M_1\ne\emptyset$. This is investigated in the last part of the paper. Suppose $P^*$ is a probability measure with compact support and define
$$\mathcal{U}=\bigl\{Q\in\mathbb{P}:Q\ll P^*\text{ and }(\pi_1,\ldots,\pi_n)\text{ is a }Q\text{-martingale}\bigr\}.$$
Then, $M_1\ne\emptyset$ if and only if
\begin{gather*}
\sum_{i=1}^n\mu_i(f_i)\ge\inf_{Q\in\mathcal{U}}Q\Bigl(\,\sum_{i=1}^nf_i\circ\pi_i\Bigr)
\end{gather*}
for all bounded Borel functions $f_1,\ldots,f_n:\mathbb{R}\rightarrow\mathbb{R}$. This result admits a $\sigma$-additive version as well. In fact, if $P^*$ has discrete marginals (except possibly one), the above condition implies the existence of $P\in\Gamma$ such that $P\ll P^*$ and $(\pi_1,\ldots,\pi_n)$ is a $P$-martingale.

\section{Preliminaries}\label{sect2}

\noindent For any measurable space $(S,\mathcal{E})$, we denote by $B(S,\mathcal{E})$ the set of bounded $\mathcal{E}$-measurable functions $f:S\rightarrow\mathbb{R}$. If $P,\,Q\in\mathbb{P}$, we write $P\ll Q$ to mean that $P(A)=0$ whenever $A\in\mathcal{A}$ and $Q(A)=0$. Moreover, if $f_i:\Omega_i\rightarrow\mathbb{R}$, the map $\sum_if_i\circ\pi_i$ is denoted by $\oplus_if_i$. Hence, $\oplus_if_i$ is the function on $\Omega$ defined by
$$\bigl(\oplus_{i=1}^nf_i\bigr)(\omega)=\sum_{i=1}^nf_i(\omega_i)\quad\quad\text{for all }\omega=(\omega_1,\ldots,\omega_n)\in\Omega.$$

\medskip

\subsection{Integrals with respect to f.a.p.'s}\label{e4n8h5} Let $P\in\mathbb{P}$ and $f:\Omega\rightarrow\mathbb{R}$ a real-valued $\mathcal{A}$-measurable function. If $f$ is simple, $\int f\,dP$ is defined in the usual way. If $f$ is bounded, $f$ is the uniform limit of a sequence $f_k$ of simple functions, and we let $\int f\,dP=\lim_k\int f_k\,dP$. If $f\ge 0$, then $f$ is $P$-integrable if and only if
$$\sup_k\int f\wedge k\,dP<\infty,$$
and we let
\begin{gather}\label{n6t4}
\int f\,dP=\sup_k\int f\wedge k\,dP.
\end{gather}
In general, $f$ is $P$-integrable if and only if $f^+$ and $f^-$ are both $P$-integrable, or equivalently $\int\abs{f}\,dP<\infty$, and in this case $\int f\,dP=\int f^+\,dP-\int f^-\,dP$. Observe now that equation \eqref{n6t4} makes sense whenever $0\le f\le\infty$, merely $\int f\,dP=\infty$ if $f$ is real-non-negative but not $P$-integrable or if $P(f=\infty)>0$. Hence, $\int f\,dP$ is always defined by \eqref{n6t4} whenever $f$ is $\mathcal{A}$-measurable and takes values in $[0,\infty]$. In a nutshel, this is a concise summary on finitely additive integration (when the integrand $f$ is measurable with respect to a $\sigma$-field). For more on this subject, we refer to \cite{BRBR}. As already noted, we will use the notation $P(f)=\int f\,dP$ whenever $\int f\,dP$ is well defined.

\subsection{Finitely additive martingales}\label{nj887c4d6h} Let $(\Omega_i,\mathcal{A}_i)=(\mathbb{R},\mathcal{B}(\mathbb{R}))$ for all $i\in I$. Define a class $H$ of real-valued functions on $\Omega=\mathbb{R}^n$ as
$$H=\{\pi_1\}\cup\bigl\{\bigl(\pi_{i+1}-\pi_i\bigr)\,g(\pi_1,\ldots,\pi_i):1\le i<n,\,\,g\in B(\mathbb{R}^i,\mathcal{B}(\mathbb{R}^i))\bigr\}.$$
If $P$ is a probability measure on $\mathcal{A}$, then $(\pi_1,\ldots,\pi_n)$ is a $P$-martingale if

\begin{itemize}

\item[(a)] $P(\abs{\pi_i})<\infty$ for all $i\in I$,

\item[(b)] $P(\pi_1)=0$ and
$$E_P\bigl(\pi_{i+1}\mid \pi_1,\ldots,\pi_i\bigr)=\pi_i,\,\,P\text{-a.s., for all }i<n.$$

\end{itemize}

\noindent However, under (a), condition (b) is equivalent to

\medskip

\begin{itemize}

\item[(c)] $P(f)=0$ for each $f\in H$.

\end{itemize}

\medskip

\noindent Therefore, if $P\in\mathbb{P}$ is a f.a.p. on $\mathcal{A}$, then $(\pi_1,\ldots,\pi_n)$ is said to be a $P$-martingale if conditions (a) and (c) are satisfied; see  \cite{BPR2013}-\cite{BPR2015}.

\subsection{The product topology on $[0,1]^\mathcal{A}$}\label{b4rt6} Let $[0,1]^\mathcal{A}$ be the collection of all functions $P:\mathcal{A}\rightarrow [0,1]$. When equipped with the product topology, $[0,1]^\mathcal{A}$ is a compact Hausdorff space and convergence is setwise convergence. Precisely, a net $(P_\alpha)$ in $[0,1]^\mathcal{A}$ converges to $P\in [0,1]^\mathcal{A}$ if and only if
$$P(A)=\lim_\alpha P_\alpha(A)\text{ for all }A\in\mathcal{A}.$$
In particular, $\mathbb{P}$ and $M$ are compact subsets of $[0,1]^\mathcal{A}$. Moreover, if $f\in B(\Omega,\mathcal{A})$, the map $P\mapsto P(f)$ is continuous on $\mathbb{P}$. In the sequel, $[0,1]^\mathcal{A}$ {\em is given the product topology and all its subsets are equipped with the corresponding relative topology}.

\section{Finitely additive couplings}\label{gu4n77y}

\noindent In this section, $c:\Omega\rightarrow [0,\infty]$ is an $\mathcal{A}$-measurable map (the value $\infty$ is admissible for $c$). Recall also that $\mathbb{P}$ denotes the set of all f.a.p.'s on $\mathcal{A}$ and $M$ the collection of those $P\in\mathbb{P}$ with marginals $\mu_1,\ldots,\mu_n$.

\medskip

\subsection{Results}\label{um7m} We begin with a preliminary lemma. Fix a collection $F$ of real-valued $\mathcal{A}$-measurable functions on $\Omega$ and define
$$\mathcal{K}(F)=\bigl\{P\in M:P(\abs{f})<\infty\text{ and }P(f)=0\text{ for all }f\in F\bigr\}.$$

\begin{lem}\label{j91s4f}
$\mathcal{K}(F)$ is compact provided
\begin{gather}\label{b7t9}
\lim_k\sup_{P\in M}P\bigl\{\abs{f}\,1(\abs{f}>k)\bigr\}=0\quad\quad\text{for each }f\in F.
\end{gather}
Moreover, $P\mapsto P(c)$ is a lower semicontinuous map from $\mathbb{P}$ into $[0,\infty]$.
\end{lem}

\begin{proof}
To make the notation easier, write $\mathcal{K}$ instead of $\mathcal{K}(F)$. Since $\mathcal{K}\subset M$ and $M$ is compact, it suffices to show that $\mathcal{K}$ is closed. Let $(P_\alpha)$ be a net in $\mathcal{K}$ such that $P_\alpha\rightarrow P$ for some $P\in [0,1]^\mathcal{A}$. Since $M$ is closed, $P\in M$.

\medskip

\noindent Let $f\in F$. We have to show that $P(\abs{f})<\infty$ and $P(f)=0$. By \eqref{b7t9}, there is $k$ such that $Q\bigl\{\abs{f}\,1(\abs{f}>k)\bigr\}\le 1$ for each $Q\in M$. Hence, $P\in M$ implies
$$P(\abs{f})\le k+P\bigl\{\abs{f}\,1(\abs{f}>k)\bigr\}\le k+1.$$
Next, for fixed $\alpha$, one obtains $P_\alpha(f)=0$ (due to $P_\alpha\in\mathcal{K}$) and
\begin{gather*}0=P_\alpha(f)=P_\alpha\bigl\{f\,1(\abs{f}\le k)\bigr\}+P_\alpha\bigl\{f\,1(\abs{f}>k)\bigr\}
\\\le P_\alpha\bigl\{f\,1(\abs{f}\le k)\bigr\}+\sup_{Q\in M}Q\bigl\{\abs{f}\,1(\abs{f}>k)\bigr\}\quad\quad\text{for all }k.
\end{gather*}
Moreover, for fixed $k$,
$$\lim_\alpha P_\alpha\bigl\{f\,1(\abs{f}\le k)\bigr\}=P\bigl\{f\,1(\abs{f}\le k)\bigr\}.$$
It follows that
\begin{gather*}
0\le P\bigl\{f\,1(\abs{f}\le k)\bigr\}+\sup_{Q\in M}Q\bigl\{\abs{f}\,1(\abs{f}>k)\bigr\}
\le P(f)+2\,\sup_{Q\in M}Q\bigl\{\abs{f}\,1(\abs{f}>k)\bigr\}.
\end{gather*}
Hence, $P(f)\ge 0$ because of \eqref{b7t9}. Similarly, one obtains $P(f)\le 0$. Thus $P(f)=0$, and this proves that $\mathcal{K}$ is closed. Finally, to show that $P\mapsto P(c)$ is lower semicontinuous on $\mathbb{P}$, it suffices to note that $P\mapsto P(c\wedge k)$ is continuous for fixed $k$. Therefore, lower semicontinuity of $P\mapsto P(c)$ follows from
$$P(c)=\sup_kP(c\wedge k)\quad\quad\text{for all }P\in\mathbb{P}.$$
\end{proof}

\noindent Condition \eqref{b7t9} is a form of uniform integrability. Since
$$P\bigl\{\abs{f}\,1(\abs{f}>k)\bigr\}\le k^{-\epsilon}\,P\bigl\{\abs{f}^{1+\epsilon}\bigr\},$$
a sufficient condition for \eqref{b7t9} is that, for each $f\in F$, there is $\epsilon>0$ such that $\sup_{P\in M}P\bigl\{\abs{f}^{1+\epsilon}\bigr\}<\infty$. In particular, condition \eqref{b7t9} holds whenever $F\subset B(\Omega,\mathcal{A})$. In any case, since $P\mapsto P(c)$ is lower semicontinuous, one obtains
\begin{gather}\label{cb9u4}
P(c)=\inf_{Q\in K}Q(c)\quad\quad\text{for some }P\in K
\end{gather}
provided $K\subset M$ is compact and non-empty. Thus, with reference to problem (i), the $\inf$ is attained for {\em any} non-negative $\mathcal{A}$-measurable cost function $c$ provided $\Gamma$ is replaced by a non-empty compact $K\subset M$. The next result highlights some meaningful special cases.

\begin{thm}\label{v77h3s1z}
Let $P^*\in\mathbb{P}$ and
$$K_1=\bigl\{P\in M:P\ll P^*\bigr\}.$$
If $(\Omega_i,\mathcal{A}_i)=(\mathbb{R},\mathcal{B}(\mathbb{R}))$ for each $i\in I$, define also
\begin{gather*}
K_2=\bigl\{P\in M:(\pi_1,\ldots,\pi_n)\text{ is a }P\text{-martingale}\bigr\}\quad\quad\text{and}
\\K_3=\bigl\{P\in M:P\ll P^*\text{ and }(\pi_1,\ldots,\pi_n)\text{ is a }P\text{-martingale}\bigr\}.
\end{gather*}
Then, equation \eqref{cb9u4} holds with $K=M$. Moreover, for each $j=1,2,3$, equation \eqref{cb9u4} holds with $K=K_j$ provided $K_j\ne\emptyset$.
\end{thm}

\begin{proof}
As noted in Subsection \ref{b4rt6}, $M$ is compact (and obviously non-empty). Moreover, $K_1=\mathcal{K}(F)$ where
$$F=\bigl\{1_A:A\in\mathcal{A},\,P^*(A)=0\bigr\}.$$
Since $F\subset B(\Omega,\mathcal{A})$, condition \eqref{b7t9} holds. Hence, $K_1$ is compact.

\medskip

\noindent Next, we let $(\Omega_i,\mathcal{A}_i)=(\mathbb{R},\mathcal{B}(\mathbb{R}))$ for all $i\in I$ and we prove that $K_2$ is compact (provided it is non-empty). To this end, we first note that $K_2=\mathcal{K}(H)$ where
\begin{gather*}
H=\{\pi_1\}\cup\bigl\{\bigl(\pi_{i+1}-\pi_i\bigr)\,g(\pi_1,\ldots,\pi_i):1\le i<n,\,\,g\in B(\mathbb{R}^i,\mathcal{B}(\mathbb{R}^i))\bigr\};
\end{gather*}
see Subsection \ref{nj887c4d6h}. Hence, it suffices to check condition \eqref{b7t9} for each $f\in H$.

\medskip

\noindent Since $K_2\neq\emptyset$, there is $Q\in K_2$ and this implies
$$\int\abs{x}\,\mu_i(dx)=Q(\abs{\pi_i})<\infty\quad\quad\text{for each }i\in I.$$
Let $f\in H$. If $f=\pi_1$, then
$$P\bigl\{\abs{\pi_1}\,1(\abs{\pi_1}>n)\bigr\}=\int\abs{x}\,1(\abs{x}>n)\,\mu_1(dx)\quad\quad\text{for each }P\in M.$$
Hence, \eqref{b7t9} holds since $\mu_1$ is $\sigma$-additive and $\int\abs{x}\,\mu_1(dx)<\infty$. Suppose now that
$$f=\bigl(\pi_{i+1}-\pi_i\bigr)\,g(\pi_1,\ldots,\pi_i)$$
for some $1\le i<n$ and some $g\in B(\mathbb{R}^i,\mathcal{B}(\mathbb{R}^i))$. To prove \eqref{b7t9}, it can be assumed $\sup\abs{g}\le 1$. Fix $\epsilon>0$, and take a constant $a>0$ such that
$$\int\abs{x}\,1(\abs{x}> a)\,\mu_i(dx)<\epsilon\quad\quad\text{for each }i\in I.$$
Then, for each $P\in M$, one obtains
\begin{gather*}
P\bigl\{\abs{f}\,1(\abs{f}>2a)\bigr\}\le P\bigl\{\abs{\pi_{i+1}-\pi_i}\,1(\abs{\pi_{i+1}-\pi_i}>2a)\bigr\}
\\\le P\Bigl\{\Bigl(\abs{\pi_{i+1}}+\abs{\pi_i}\Bigr)\,\Bigl(1(\abs{\pi_{i+1}}>a)+1(\abs{\pi_i}>a)\Bigr)\Bigr\}
\\=P\bigl\{\abs{\pi_i}\,1(\abs{\pi_i}>a)\bigr\}+P\bigl\{\abs{\pi_{i+1}}\,1(\abs{\pi_{i+1}}>a)\bigr\}\,+
\\+\,P\bigl\{\abs{\pi_{i+1}}\,1(\abs{\pi_i}>a)\bigr\}+P\bigl\{\abs{\pi_i}\,1(\abs{\pi_{i+1}}>a)\bigr\}.
\end{gather*}
In addition,
\begin{gather*}
P\bigl\{\abs{\pi_i}\,1(\abs{\pi_{i+1}}>a)\bigr\}\le a\,P(\abs{\pi_{i+1}}>a)+P\bigl\{\abs{\pi_i}\,1(\abs{\pi_i}> a)\bigr\}
\\\le P\bigl\{\abs{\pi_{i+1}}\,1(\abs{\pi_{i+1}}> a)\bigr\}+P\bigl\{\abs{\pi_i}\,1(\abs{\pi_i}> a)\bigr\}
\end{gather*}
and similarly
$$P\bigl\{\abs{\pi_{i+1}}\,1(\abs{\pi_i}>a)\bigr\}\le P\bigl\{\abs{\pi_i}\,1(\abs{\pi_i}> a)\bigr\}+P\bigl\{\abs{\pi_{i+1}}\,1(\abs{\pi_{i+1}}> a)\bigr\}.$$
Therefore,
\begin{gather*}
P\bigl\{\abs{f}\,1(\abs{f}>2a)\bigr\}\le 6\,\max_{i\in I}\,P\bigl\{\abs{\pi_i}\,1(\abs{\pi_i}>a)\bigr\}.
\end{gather*}
Finally, since $P\bigl\{\abs{\pi_i}\,1(\abs{\pi_i}>a)\bigr\}=\int\abs{x}\,1(\abs{x}> a)\,\mu_i(dx)$ for each $P\in M$, one obtains
\begin{gather*}
\sup_{P\in M}P\bigl\{\abs{f}\,1(\abs{f}>k)\bigr\}\le 6\,\max_{i\in I}\,\int\abs{x}\,1(\abs{x}> a)\,\mu_i(dx)<6\,\epsilon
\end{gather*}
whenever $k\ge 2a$. Hence, condition \eqref{b7t9} holds and this proves that $K_2$ is compact. Finally, $K_3$ is compact since $K_3=K_1\cap K_2$.
\end{proof}

\medskip

\noindent Whether or not the sets $K_1$ and $K_3$ are non-empty is investigated in Section \ref{h81a3a}.

\medskip

\noindent Let us turn to problem (iii). Similarly to (i), everything goes smoothly if $\Gamma$ is replaced by $M$.

\begin{thm}\label{h9k33d}
Duality always holds with $M$ in the place of $\Gamma$, namely
\begin{gather*}
\inf_{Q\in M}Q(c)=\sup_{f_1,\ldots,f_n}\,\sum_{i=1}^n\mu_i(f_i)
\end{gather*}
where $\sup$ is over the $n$-tuple $(f_1,\ldots,f_n)$ such that $f_i\in L_1(\mu_i)$ for each $i\in I$ and $\oplus_{i=1}^nf_i\le c$.
\end{thm}

\begin{proof}
If $Q\in M$, $f_i\in L_1(\mu_i)$ for all $i\in I$ and $\oplus_{i=1}^nf_i\le c$, then
$$Q(c)\ge Q\Bigl(\,\oplus_{i=1}^nf_i\Bigr)=\sum_{i=1}^n Q\bigl(f_i\circ\pi_i\bigr)=\sum_{i=1}^n\mu_i(f_i).$$
Hence, it suffices to show that
$$\sup_{f_1,\ldots,f_n}\,\sum_{i=1}^n\mu_i(f_i)\ge\inf_{Q\in M}Q(c).$$

\medskip

\noindent Let
$$D=\bigl\{\,\oplus_{i=1}^nf_i:f_i\in B(\Omega_i,\mathcal{A}_i),\,i\in I\bigr\}.$$
Then, $D$ is a linear subspace of $B(\Omega,\mathcal{A})$ including the constants. If
$$\oplus_{i=1}^nf_i=\oplus_{i=1}^ng_i\quad\quad\text{with }f_i,\,g_i\in B(\Omega_i,\mathcal{A}_i)\text{ for all }i,$$
then $f_i=g_i+a_i$ for all $i\in I$ where $a_1,\ldots,a_n$ are constants satisfying $\sum_{i=1}^na_i=0$. Thus,
$$\sum_{i=1}^n\mu_i(f_i)=\sum_{i=1}^n\mu_i(g_i).$$
As a consequence, for each $\oplus_{i=1}^nf_i\in D$, one can define
$$T\Bigl(\oplus_{i=1}^nf_i\Bigr)=\sum_{i=1}^n\mu_i(f_i).$$
Then, $T:D\rightarrow\mathbb{R}$ is a linear positive functional such that $T(1)=1$.

\medskip

\noindent For each integer $k\ge 1$, define
$$D_k=\bigl\{f\in D:f\le c\wedge k\bigr\}.$$
By Hahn-Banach theorem, $T$ can be extended to a linear positive functional $T_k$ on $B(\Omega,\mathcal{A})$ such that
$$T_k(c\wedge k)=\sup_{f\in D_k}T(f);$$
see e.g. \cite[Lemma 2]{BR2021}. Moreover, by standard arguments on the theory of coherence (see \cite[Sect. 2]{BR2021} again), there is a f.a.p. $P_k\in\mathbb{P}$ such that
$$T_k(f)=\int f\,dP_k\quad\quad\text{for all }f\in B(\Omega,\mathcal{A}).$$
For $i\in I$ and $B\in\mathcal{A}_i$, one obtains
$$P_k(\pi_i\in B)=T_k\bigl(1(\pi_i\in B)\bigr)=T\bigl(1(\pi_i\in B)\bigr)=\mu_i(B).$$
Hence, $P_k\in M$. Moreover, letting
$$D_\infty=\bigl\{f\in D:f\le c\bigr\},$$
one obtains
$$\sup_{f\in D_\infty}T(f)\ge\sup_{f\in D_k}T(f)=\int c\wedge k\,dP_k=P_k(c\wedge k).$$

\medskip

\noindent Next, consider the sequence $(P_k:k\ge 1)$. Since $M$ is compact,
$$P=\lim_\alpha P_{k_\alpha}$$
for some $P\in M$ and some subnet $(P_{k_\alpha})$ of $(P_k)$. Given an integer $m\ge 1$, there is $\alpha^*$ such that $k_\alpha\ge m$ whenever $\alpha\ge\alpha^*$. Therefore,
$$\sup_{f\in D_\infty}T(f)\ge P_{k_\alpha}(c\wedge k_\alpha)\ge P_{k_\alpha}(c\wedge m)\quad\quad\text{for each }\alpha\ge\alpha^*.$$
It follows that
$$\sup_{f\in D_\infty}T(f)\ge\lim_\alpha P_{k_\alpha}(c\wedge m)=P(c\wedge m).$$
Since $P\in M$, this in turn implies
$$\sup_{f\in D_\infty}T(f)\ge\sup_mP(c\wedge m)=P(c)\ge\inf_{Q\in M}Q(c).$$

\medskip

\noindent Finally, fix $\epsilon>0$ and $f_i\in L_1(\mu_i)$ such that $\oplus_{i=1}^nf_i\le c$. It is easily seen that there are $g_1,\ldots,g_n$ satisfying
$$g_i\in B(\Omega_i,\mathcal{A}_i)\text{ for each }i\in I,\quad\oplus_{i=1}^ng_i\le c,\quad\sum_{i=1}^n\mu_i(g_i)+\epsilon > \sum_{i=1}^n\mu_i(f_i).$$
Therefore,
$$\sup_{f\in D_\infty}T(f)=\sup_{f_1,\ldots,f_n}\,\sum_{i=1}^n\mu_i(f_i)$$
where the $\sup$ on the right is over the $n$-tuple $(f_1,\ldots,f_n)$ such that $f_i\in L_1(\mu_i)$ for each $i\in I$ and $\oplus_{i=1}^nf_i\le c$. This concludes the proof.
\end{proof}

\medskip

\noindent Finally, we focus on problem (ii). Let
$$L=\bigl\{\oplus_{i=1}^nf_i:f_i\in L_1(\mu_i),\,i\in I\bigr\}.$$

\begin{thm}\label{uf6fa1}
Suppose $c\le f^*$ for some $f^*\in L$. Then, for each $P\in M$,
$$P(c)=\inf_{Q\in M}Q(c)$$
if and only if
\begin{gather}\label{qq28n4}
\text{There is }f\in L\text{ such that }f\le c\text{ and }P(c>f+\epsilon)=0\text{ for all }\epsilon>0.
\end{gather}
\end{thm}

\medskip

\noindent It is worth noting that, since $P\in M$ is not necessarily $\sigma$-additive, $P(c>f+\epsilon)=0$ for all $\epsilon>0$ does not imply $P(c>f)=0$.

\begin{proof}[Proof of Theorem \ref{uf6fa1}]
First note that, for all $Q\in M$ and $g=\oplus_{i=1}^ng_i\in L$,
$$Q(g)=\sum_{i=1}^nQ\bigl(g_i\circ\pi_i\bigr)=\sum_{i=1}^n\mu_i(g_i).$$
Hence, for fixed $g\in L$, the map $Q\mapsto Q(g)$ is constant on $M$. A second (and basic) remark is that, since $c\le f^*$ for some $f^*\in L$, there is $f:\Omega\rightarrow\mathbb{R}$ such that
\begin{gather}\label{v03e7jk8}
f\in L,\,\,f\le c\quad\text{and}\quad Q(f)=\sup_{g\in L,\,g\le c}Q(g)\quad\text{for each }Q\in M;
\end{gather}
see Theorem (2.21) of \cite{KEL}.

\medskip

\noindent Having noted these facts, fix $P\in M$. Suppose $P(c)=\inf_{Q\in M}Q(c)$ and take a function $f$ satisfying \eqref{v03e7jk8}. Then,
$$P(c)=\inf_{Q\in M}Q(c)=\sup_{g\in L,\,g\le c}P(g)=P(f)$$
where the second equality is due to Theorem \ref{h9k33d}. Hence, condition \eqref{qq28n4} follows from $f\le c$ and $P(c-f)=0$. Conversely, if $f$ is any function satisfying \eqref{qq28n4}, one obtains
$$P(c)=P(f)=Q(f)\le Q(c)\quad\quad\text{for each }Q\in M.$$
\end{proof}

\medskip

\noindent In the standard framework (i.e., with $\Gamma$ in the place of $M$), to settle problems (i)-(ii)-(iii), one needs some conditions. To fix ideas, suppose the $\Omega_i$ are separable metric spaces and $\mathcal{A}_i$ the corresponding Borel $\sigma$-fields. Then, for the $\inf$ in problem (i) to be attained, $c$ should be lower semicontinuous and each $\mu_i$ a tight probability measure. As regards (iii), duality holds if $0\le c<\infty$ and the $\Omega_i$ are Polish spaces. Moreover, in case the $(\Omega_i,\mathcal{A}_i)$ are arbitrary measurable spaces, duality holds if $c\le f^*$ for some $f^*\in L$ and all but one the $\mu_i$ are perfect probability measures. Similarly, as regards problem (ii). An analogous of Theorem \ref{uf6fa1} holds, with $\Gamma$ in the place of $M$, provided all but one the $\mu_i$ are perfect probability measures. See e.g. \cite{AGS}, \cite{BS2011}, \cite{EKSO}, \cite{KEL}, \cite{RRLIBRO}, \cite{RAMRUS1995}, \cite{RIGO}, \cite{RUSC1996}, \cite{VILLA}.

\medskip

\noindent Thus, in a sense, this Subsection can be summarized by stating that all the above conditions are superfluous if the couplings are allowed to be finitely additive.

\medskip

\subsection{Two examples} In this Subsection, we let $n=2$ and
$$(\Omega_i,\mathcal{A}_i,\mu_i)=([0,1],\mathcal{B}([0,1]),m)$$
for each $i$, where $m$ is the Lebesgue measure. In addition, $T$ denotes the element of $\Gamma$ supported by the diagonal, that is
$$T(A)=m\bigl\{x\in [0,1]:(x,x)\in A\bigr\}\quad\quad\text{for each }A\in\mathcal{A}=\mathcal{B}([0,1]^2).$$

\medskip

\begin{ex}\label{f90h4qj7}
We first note that there are two f.a.p.'s $P$ and $\widetilde{P}$ on $\mathcal{A}$ such that
$$P,\,\widetilde{P}\in M\quad\text{and}\quad P(\pi_1>\pi_2)=\widetilde{P}(\pi_1<\pi_2)=1.$$
To prove this fact, denote by $\mathcal{R}$ the field generated by the measurable rectangles $A_1\times A_2$ with $A_1,\,A_2\in\mathcal{B}([0,1])$. It is not hard to see that
\begin{gather}\label{w71mrgh6}
T(A)=1\quad\text{whenever}\quad A\in\mathcal{R}\text{ and }A\supset\{\pi_1>\pi_2\}.
\end{gather}
Because of \eqref{w71mrgh6}, the restriction $T|\mathcal{R}$ can be extended to a f.a.p. $P$ on $\mathcal{A}$ such that $P(\pi_1>\pi_2)=1$; see e.g. \cite[Theo. 3.3.3]{BRBR}. Moreover, since $T\in\Gamma$ and $P=T$ on $\mathcal{R}$, one obtains $P\in M$. The existence of $\widetilde{P}$ can be proved by the same argument.

\medskip

\noindent Next, define
$$c=1(\pi_1\le\pi_2).$$
In this case, duality holds for $\Gamma$. Therefore,
$$\inf_{Q\in\Gamma}Q(c)=\sup_{f_1,f_2}\,\sum_{i=1}^2m(f_i)\le P(c)=P(\pi_1\le\pi_2)=0,$$
where $\sup$ is over the pairs $(f_1,f_2)$ where $f_1$ and $f_2$ are $m$-integrable functions such that $f_1(x)+f_2(y)\le c(x,y)$ for all $(x,y)\in [0,1]^2$. However, for each $Q\in\Gamma$,
$$Q(\pi_1-\pi_2)=Q(\pi_1)-Q(\pi_2)=0\quad\Rightarrow\quad Q(c)=Q(\pi_1\le\pi_2)>0.$$
Hence, the $\inf$ in problem (i) is not attained. On the contrary, by Theorem \ref{v77h3s1z}, the $\inf$ is attained if $\Gamma$ is replaced by $M$. In fact,
$$\inf_{Q\in M}Q(c)=0=P(c).$$

\medskip

\noindent Next, for all $(x,y)$, define
$$c(x,y)=\infty\text{ if }x<y,\quad c(x,y)=1\text{ if }x=y,\quad c(x,y)=0\text{ if }x>y.$$
If $Q\in\Gamma$ and $Q(c)<\infty$, then $Q(\pi_1=\pi_2)=1$. Therefore, $\inf_{Q\in\Gamma}Q(c)=1$. However, as shown in \cite[Ex. 4.1]{BS2011}, duality fails for $\Gamma$. In fact, if $f_1$ and $f_2$ are $m$-integrable functions such that $f_1(x)+f_2(y)\le c(x,y)$ for all $(x,y)$, then
\begin{gather*}
m(f_1)+m(f_2)=\lim_{\epsilon\rightarrow 0}\,\Bigl\{\int_\epsilon^1f_1(x)\,dx+\int_0^{1-\epsilon}f_2(x)\,dx\Bigr\}
\\=\lim_{\epsilon\rightarrow 0}\int_0^{1-\epsilon}\bigl\{f_1(x+\epsilon)+f_2(x)\bigr\}\,dx\le\lim_{\epsilon\rightarrow 0}\int_0^{1-\epsilon}c(x+\epsilon,x)\,dx=0.
\end{gather*}
Once again, because of Theorem \ref{h9k33d}, duality holds if $\Gamma$ is replaced by $M$. In fact, since $\{c=0\}=\{\pi_1>\pi_2\}$, one obtains
$$\sup_{f_1,f_2}\,\sum_{i=1}^2m(f_i)=0=P(c)=\inf_{Q\in M}Q(c).$$
\end{ex}

\medskip

\begin{ex}\label{ab6ym1l}
Let $P$ and $\widetilde{P}$ be the f.a.p.'s introduced in Example \ref{f90h4qj7} and
$$K=\bigl\{Q\in M:(\pi_1,\pi_2)\text{ is a }Q\text{-martingale}\bigr\}.$$
Then, $P,\,\widetilde{P}\in K$. In fact, if $f(x,y)=g(x)\,(y-x)$ for all $(x,y)$ and some bounded Borel function $g:[0,1]\rightarrow\mathbb{R}$, then
\begin{gather*}
\abs{P(f)}\le P(\abs{f})\le \sup\abs{g}\,P(\abs{\pi_2-\pi_1})
\\=\sup\abs{g}\,P(\pi_1-\pi_2)=\sup\abs{g}\,\bigl\{P(\pi_1)-P(\pi_2)\bigr\}=0.
\end{gather*}
Hence $P(f)=0$, and similarly $\widetilde{P}(f)=0$.

\medskip

\noindent The fact that $P,\,\widetilde{P}\in K$ suggests two (related) remarks.

\medskip

\begin{itemize}

\item $K$ contains various elements in addition to $T$. For instance,
$$a\,P+b\,\widetilde{P}+(1-a-b)\,T\in K$$
whenever $a,\,b\ge 0$ and $a+b\le 1$. On the contrary, the only member of $K\cap\Gamma$ is $T$.

\medskip

\item Let $c=1(\pi_1=\pi_2)$. Arguing as in \cite[Ex. 8.1]{BNT2017}, it can be shown that
$$m(f_1)+m(f_2)\le 0$$
whenever $f_1$ and $f_2$ are $m$-integrable and satisfy
\begin{gather}\label{ga3p}
f_1(x)+f_2(y)\le c(x,y)-g(x)\,(y-x)
\end{gather}
for all $(x,y)$ and some bounded Borel function $g$. Since $K\cap\Gamma=\{T\}$, one obtains
$$\inf_{Q\in K\cap\Gamma}Q(c)=T(c)=1>0=\sup_{f_1,f_2}\bigl\{m(f_1)+m(f_2)\bigr\}$$
where $\sup$ is over the pairs $(f_1,f_2)$ of $m$-integrable functions satisfying condition \eqref{ga3p}. Hence, in the standard framework, there is a duality gap in the martingale problem. To avoid this gap, a suitable almost sure formulation of the dual problem is to be adopted; see \cite{BNT2017} again. On the contrary, the gap does not arise if $\Gamma$ is replaced by $M$ (so that $K\cap\Gamma$ is replaced by $K$). In fact,
$$0\le\inf_{Q\in K}Q(c)\le P(c)=P(\pi_1=\pi_2)=0.$$

\end{itemize}

\medskip

\noindent A conjecture is that, with $M$ in the place of $\Gamma$, the absence of duality gap in the martingale problem is a general fact, and not a lucky feature of this example. Another conjecture is that some of the stability results, which fail in the standard framework when $\Omega_i=\mathbb{R}^d$, are valid in a finitely additive setting; see \cite{BAVE}, \cite{BJMP}, \cite{BRUJUIL}.

\end{ex}

\medskip

\section{Existence of f.a.p.'s and probability measures\\satisfying certain conditions}\label{h81a3a}

\noindent For Theorem \ref{v77h3s1z} to apply, certain collections $K_j$ of f.a.p.'s should be non-empty. We now give conditions for $K_j\ne\emptyset$. By the same argument, we also obtain conditions for the existence of certain probability measures. This section is based on ideas from \cite[Sect. 3]{BPRSPIZ2015} which in turn was inspired by Strassen \cite{STR}.

\medskip

\noindent Let $\mathcal{U}\subset\mathbb{P}$ be any collection of f.a.p.'s on $\mathcal{A}$. If $P\in M\cap\mathcal{U}$ and $f_i\in B(\Omega_i,\mathcal{A}_i)$ for each $i\in I$, one trivially obtains
\begin{gather*}
\sum_{i=1}^n\mu_i(f_i)=P\Bigl(\oplus_{i=1}^nf_i\Bigr)\ge\inf_{Q\in\mathcal{U}}Q\Bigl(\oplus_{i=1}^nf_i\Bigr)
\end{gather*}
where the equality is due to $P\in M$ and the inequality to $P\in\mathcal{U}$. Hence,
\begin{gather}\label{j92x1ly6}
\sum_{i=1}^n\mu_i(f_i)\ge\inf_{Q\in\mathcal{U}}Q\Bigl(\oplus_{i=1}^nf_i\Bigr)\quad\quad\text{whenever }f_i\in B(\Omega_i,\mathcal{A}_i)\text{ for all }i\in I.
\end{gather}
Our next goal is proving that, for some choices of $\mathcal{U}$, condition \eqref{j92x1ly6} is actually equivalent to $M\cap\mathcal{U}\ne\emptyset$. In what follows, we adopt the usual convention $$\inf\emptyset=\infty.$$
Therefore, condition \eqref{j92x1ly6} is trivially false if $\mathcal{U}=\emptyset$.

\medskip

\begin{thm}\label{f6h99k33ws}
Let $P^*\in\mathbb{P}$ and $F$ a collection of $\mathcal{A}$-measurable functions $f:\Omega\rightarrow\mathbb{R}$. If $\mathcal{U}=\bigl\{Q\in\mathbb{P}:Q\ll P^*\bigr\}$, condition \eqref{j92x1ly6} is equivalent to $M\cap\mathcal{U}\ne\emptyset$. Moreover, if
$$\mathcal{U}=\bigl\{Q\in\mathbb{P}:Q\ll P^*,\,Q(\abs{f})<\infty\text{ and }Q(f)=0\text{ for each }f\in F\bigr\},$$
then \eqref{j92x1ly6} is equivalent to $M\cap\mathcal{U}\ne\emptyset$ provided
\begin{gather}\label{ad3339jgt}
\lim_k\sup_{Q\in\mathcal{U}}Q\bigl\{\abs{f}\,1(\abs{f}>k)\bigr\}=0\quad\text{for all }f\in F.
\end{gather}
\end{thm}

\begin{proof}
Let $\mathcal{U}\subset\mathbb{P}$. As proved above, condition \eqref{j92x1ly6} holds if $M\cap\mathcal{U}\ne\emptyset$. Hence, suppose condition \eqref{j92x1ly6} holds. As in the proof of Theorem \ref{h9k33d}, define $D$ to be the collection of functions $f$ of the form $f=\oplus_{i=1}^nf_i$ where $f_i\in B(\Omega_i,\mathcal{A}_i)$ for each $i\in I$. For $f=\oplus_{i=1}^nf_i\in D$, let
$$X_f(Q)=\sum_{i=1}^n\mu_i(f_i)-Q(f)\quad\quad\text{for all }Q\in\mathcal{U}.$$
Then, $X_f$ is a bounded function on $\mathcal{U}$ and $\bigl\{X_f:f\in D\bigr\}$ is a linear space. Moreover, for fixed $f\in D$, condition \eqref{j92x1ly6} yields
$$\sup_{Q\in\mathcal{U}}X_f(Q)=\sum_{i=1}^n\mu_i(f_i)-\inf_{Q\in\mathcal{U}}Q(f)\ge 0.$$
In turn, the above condition implies the existence of a f.a.p. $P_\mathcal{U}$, defined on the power set of $\mathcal{U}$, such that
$$\int X_f(Q)\,P_\mathcal{U}(dQ)=0\quad\quad\text{for each }f\in D$$
or equivalently
\begin{gather}\label{g71vf5t}
\int Q\bigl(\oplus_{i=1}^nf_i\bigr)\,P_\mathcal{U}(dQ)=\sum_{i=1}^n\mu_i(f_i)\quad\text{if }f_i\in B(\Omega_i,\mathcal{A}_i)\text{ for all }i\in I;
\end{gather}
see \cite[Sect. 2]{BR2021}. Using $P_\mathcal{U}$, define
$$P(A)=\int  Q(A)\,P_\mathcal{U}(dQ)\quad\quad\text{for all }A\in\mathcal{A}.$$
Then, $P$ is a f.a.p. on $\mathcal{A}$. For each $i\in I$ and each $B\in\mathcal{A}_i$, equation \eqref{g71vf5t} yields
$$P(\pi_i\in B)=\int  Q(\pi_i\in B)\,P_\mathcal{U}(dQ)=\mu_i(B).$$
Therefore, $P\in M$. In addition, $P\in\mathcal{U}$ if $\mathcal{U}=\bigl\{Q\in\mathbb{P}:Q\ll P^*\bigr\}$. In this case, in fact, $Q(A)=0$ for each $Q\in\mathcal{U}$ whenever $A\in\mathcal{A}$ and $P^*(A)=0$. Hence, one trivially obtains $P\ll P^*$.

\medskip

\noindent It remains to see that $P\in\mathcal{U}$ if condition \eqref{ad3339jgt} holds and
$$\mathcal{U}=\bigl\{Q\in\mathbb{P}:Q\ll P^*,\,\,\,Q(\abs{f})<\infty\text{ and }Q(f)=0\text{ for each }f\in F\bigr\}.$$
Arguing as above, it is obvious that $P\ll P^*$. Hence, we have to see that $P(\abs{f})<\infty$ and $P(f)=0$ for each $f\in F$. If $g\in B(\Omega,\mathcal{A})$, since $g$ is the uniform limit of a sequence $g_k$ of simple functions, one obtains
$$P(g)=\lim_kP(g_k)=\lim_k\int Q(g_k)\,P_\mathcal{U}(dQ)=\int Q(g)\,P_\mathcal{U}(dQ).$$
Having noted this fact, fix $f\in F$. For each integer $j\ge 1$,
\begin{gather*}
P(\abs{f})\le j+P\bigl\{\abs{f}\,1(\abs{f}>j)\bigr\}=j+\sup_kP\bigl\{\abs{f}\wedge k\,1(\abs{f}>j)\bigr\}
\\=j+\sup_k\int Q\bigl\{\abs{f}\wedge k\,1(\abs{f}>j)\bigr\}\,P_\mathcal{U}(dQ)\le j+\sup_{Q\in\mathcal{U}}Q\bigl\{\abs{f}\,1(\abs{f}>j)\bigr\}.
\end{gather*}
Hence, $P(\abs{f})<\infty$ follows from condition \eqref{ad3339jgt}. Similarly,
$$P(f)=\lim_kP\bigl\{f\,1(\abs{f}\le k)\bigr\}=\lim_k\int Q\bigl\{f\,1(\abs{f}\le k)\bigr\}\,P_\mathcal{U}(dQ).$$
For every $Q\in\mathcal{U}$, since $Q(f)=0$, one obtains
$$Q\bigl\{f\,1(\abs{f}\le k)\bigr\}=-Q\bigl\{f\,1(\abs{f}>k)\bigr\}.$$
Therefore, condition \eqref{ad3339jgt} implies again
$$P(f)=\lim_k\int Q\bigl\{-f\,1(\abs{f}>k)\bigr\}\,P_\mathcal{U}(dQ)\le\lim_k\sup_{Q\in\mathcal{U}}Q\bigl\{\abs{f}\,1(\abs{f}>k)\bigr\}=0.$$
This proves that $P(f)\le 0$. Replacing $f$ with $-f$, one also obtains $P(f)\ge 0$. Thus, $P(f)=0$ and this concludes the proof.
\end{proof}

\medskip

\noindent Theorem \ref{f6h99k33ws} applies to martingale transport.

\medskip

\begin{cor}\label{gb9971s}
Let $P^*\in\mathbb{P}$ and $(\Omega_i,\mathcal{A}_i)=(\mathbb{R},\mathcal{B}(\mathbb{R}))$ for all $i\in I$. Suppose that $P^*(A)=1$ for some compact set $A\subset\mathbb{R}^n$. Then, condition \eqref{j92x1ly6} is equivalent to $M\cap\mathcal{U}\ne\emptyset$ where
\begin{gather*}
\mathcal{U}=\bigl\{Q\in\mathbb{P}:Q\ll P^*\text{ and }(\pi_1,\ldots,\pi_n)\text{ is a }Q\text{-martingale}\bigr\}.
\end{gather*}
\end{cor}

\begin{proof}
The collection $\mathcal{U}$ can be written as in the second part of Theorem \ref{f6h99k33ws} with $F=H$, where $H$ has been introduced in Subsection \ref{nj887c4d6h}. Hence, it suffices to check condition \eqref{ad3339jgt} with $F=H$. Let
$$B=\bigcup_{i=1}^n\pi_i(A)=\bigl\{\pi_i(x):i\in I,\,x\in A\bigr\}.$$
Then, $B$ is a compact subset of $\mathbb{R}$ and
$$P^*\bigl(\pi_i\in B\text{ for each }i\in I\bigr)\ge P^*(A)=1.$$
Take $k$ such that $P^*\bigl(\max_i\abs{\pi_i}\le k\bigr)=1$. Since $Q\ll P^*$ for each $Q\in\mathcal{U}$,
$$Q(\abs{\pi_1}>k)=Q(\abs{\pi_{i+1}-\pi_i}>2k)=0\quad\quad\text{for all }Q\in\mathcal{U}\text{ and }1\le i<n.$$
Fix now $f\in H$. If $f=\pi_1$, then $Q\bigl\{\abs{\pi_1}\,1(\abs{\pi_1}>k)\bigr\}=0$ for each $Q\in\mathcal{U}$. If
$$f=\bigl(\pi_{i+1}-\pi_i\bigr)\,g(\pi_1,\ldots,\pi_i),$$
for some $i$ and $g\in B(\mathbb{R}^i,\mathcal{B}(\mathbb{R}^i))$, take an integer $j\ge 2\,k\,\sup\abs{g}$ and note that
$$Q\bigl\{\abs{f}\,1\bigl(\abs{f}>j\bigr)\bigr\}\le\sup\abs{g}\,Q\bigl\{\abs{\pi_{i+1}-\pi_i}\,1(\abs{\pi_{i+1}-\pi_i}>2k)\bigr\}=0$$
for each $Q\in\mathcal{U}$. Hence, condition \eqref{ad3339jgt} holds.
\end{proof}

\medskip

\noindent The same argument used for Theorem \ref{f6h99k33ws} and Corollary \ref{gb9971s} can be applied to prove the existence of an absolutely continuous, martingale, {\em probability measure} with given marginals. More precisely, our last result provides conditions for the existence of $P\in\Gamma$ (so that $P$ is a probability measure on $\mathcal{A}$ with marginals $\mu_1,\ldots,\mu_n$) such that $P\ll P^*$ and $(\pi_1,\ldots,\pi_n)$ is a $P$-martingale.

\medskip

\begin{thm}\label{thmsadd444}
Let $P^*$ be a probability measure on $\mathcal{A}$ and $(\Omega_i,\mathcal{A}_i)=(\mathbb{R},\mathcal{B}(\mathbb{R}))$ for all $i\in I$. Suppose $P^*$ has discrete marginals, except possibly one, and $P^*(A)=1$ for some compact $A\subset\mathbb{R}^n$. Then, condition \eqref{j92x1ly6} is equivalent to $\Gamma\cap\mathcal{U}\ne\emptyset$ where
\begin{gather*}
\mathcal{U}=\bigl\{Q\in\mathbb{P}:Q\ll P^*\text{ and }(\pi_1,\ldots,\pi_n)\text{ is a }Q\text{-martingale}\bigr\}.
\end{gather*}
\end{thm}

\begin{proof}
If $P\in\Gamma\cap\mathcal{U}$, then $P\in M\cap\mathcal{U}$, and condition \eqref{j92x1ly6} follows exactly as above. Conversely, assume condition \eqref{j92x1ly6}. Let us adopt the same notation as in the proof of Theorem \ref{f6h99k33ws}. Arguing as in such a proof, because of \eqref{j92x1ly6}, there is a f.a.p. $P_\mathcal{U}$, defined on the power set of $\mathcal{U}$, satisfying equation \eqref{g71vf5t}.

\medskip

\noindent Let $\mathcal{R}$ be the field on $\Omega$ generated by the measurable rectangles $A_1\times\ldots\times A_n$, where $A_i\in\mathcal{A}_i$ for each $i\in I$. Define
$$P_0(A)=\int Q(A)\,P_\mathcal{U}(dQ)\quad\quad\text{for all }A\in\mathcal{R}.$$
(Note that $P_0$ has been defined only on $\mathcal{R}$, not on all of $\mathcal{A}$). Because of \eqref{g71vf5t},
$$P_0(\pi_i\in B)=\mu_i(B)\quad\quad\text{for all }i\in I\text{ and }B\in\mathcal{A}_i.$$
Hence, $\mu_1,\ldots,\mu_n$ are the marginals of $P_0$. Since $\mu_1,\ldots,\mu_n$ are all $\sigma$-additive and perfect, it follows that $P_0$ is $\sigma$-additive as well; see e.g. \cite{R1996}. Let $P$ be the (only) $\sigma$-additive extension of $P_0$ to $\mathcal{A}=\sigma(\mathcal{R})$. Then, $P$ is a probability measure on $\mathcal{A}$ with marginals $\mu_1,\ldots,\mu_n$, namely, $P\in\Gamma$. Furthermore,
\begin{gather}\label{h8c4k087j2}
P(A)=P_0(A)=0\quad\text{whenever}\quad A\in\mathcal{R}\text{ and }P^*(A)=0.
\end{gather}
By Lemma 4 of \cite{BPRSPIZ2015}, since $P^*$ is a probability measure and all but one its marginals are discrete, condition \eqref{h8c4k087j2} implies $P\ll P^*$.

\medskip

\noindent It remains to see that $(\pi_1,\ldots,\pi_n)$ is a $P$-martingale, namely, $P(\abs{f})<\infty$ and $P(f)=0$ for all $f\in H$ where $H$ has been introduced in Subsection \ref{nj887c4d6h}. Since $P^*$ has compact support, there is $k$ such that $P^*\bigl(\max_i\abs{\pi_i}\le k\bigr)=1$. Since $P\ll P^*$ and $Q\ll P^*$ for each $Q\in\mathcal{U}$,
$$P\bigl(\max_i\abs{\pi_i}\le k\bigr)=Q\bigl(\max_i\abs{\pi_i}\le k\bigr)=1\quad\quad\text{for each }Q\in\mathcal{U}.$$
In particular, $P(\abs{f})<\infty$ for all $f\in H$. Let $C(\mathcal{R})$ be the class of those functions $f:\Omega\rightarrow\mathbb{R}$ such that $f_k\rightarrow f$ {\em uniformly} for some sequence $f_k$ of $\mathcal{R}$-simple functions. (Such functions are called $\mathcal{R}$-continuous in \cite{BRBR}). If $f\in C(\mathcal{R})$,
$$P(f)=\lim_kP(f_k)=\lim_kP_0(f_k)=\lim_k\int Q(f_k)\,P_\mathcal{U}(dQ)=\int Q(f)\,P_\mathcal{U}(dQ)$$
where $f_k$ is any sequence of $\mathcal{R}$-simple functions such that $f_k\rightarrow f$ uniformly. Since $\pi_1\,1(\abs{\pi_1}\le k)\in C(\mathcal{R})$ and $Q(\pi_1)=0$ for all $Q\in\mathcal{U}$, it follows that
$$P(\pi_1)=P\bigl\{\pi_1\,1(\abs{\pi_1}\le k)\bigr\}=\int Q\bigl\{\pi_1\,1(\abs{\pi_1}\le k)\bigr\}\,P_\mathcal{U}(dQ)=\int Q(\pi_1)\,P_\mathcal{U}(dQ)=0.$$
Suppose now that $f$ is of the form
\begin{gather}\label{gh71zzd4kn79}
f=\bigl(\pi_{i+1}-\pi_i\bigr)\,1(\pi_1\in A_1)\,\ldots\,1(\pi_i\in A_i),
\end{gather}
where $1\le i<n$ and $A_j\in\mathcal{A}_j$ for $j=1,\ldots,i$. On noting that
$$\tilde{f}:=f\,1\bigl(\abs{\pi_i}\le k)\,1\bigl(\abs{\pi_{i+1}}\le k)\in C(\mathcal{R}),$$
one obtains
\begin{gather*}
P(f)=P(\tilde{f})=\int Q(\tilde{f})\,P_\mathcal{U}(dQ)=\int Q(f)\,P_\mathcal{U}(dQ)=0.
\end{gather*}
To sum up, $P(\pi_1)=0$ and $P(f)=0$ for all $f$ as in \eqref{gh71zzd4kn79}. Since $P$ is a probability measure, this implies $P(f)=0$ for all $f\in H$.
\end{proof}

\medskip

\noindent We conclude this paper with two remarks.

\medskip

\noindent Firstly, it would be nice to have an analogous of Theorem \ref{thmsadd444}, as well as of the other results of this paper, with $P\sim P^*$ in the place of $P\ll P^*$, where $P\sim P^*$ means $P\ll P^*$ and $P^*\ll P$. To this end, however, the techniques of this paper seem not to work. A few partial results are in \cite{BPR2013}, \cite{BPR2015}, \cite{BPRSPIZ2015}.

\medskip

\noindent Secondly, the class $\mathcal{U}$ involved in Corollary \ref{gb9971s} and Theorem \ref{thmsadd444} may be empty. Let us consider the case where $P^*$ is a probability measure. Then, by \cite{DMW}, one obtains $\mathcal{U}\ne\emptyset$ if $P^*$ satisfies the {\em non-arbitrage} condition
$$P^*(f\ge 0)=1\quad\Rightarrow\quad P^*(f=0)=1\quad\quad\text{for each }f\text{ in the linear span of }H.$$
The non-arbitrage condition, however, is stronger than $\mathcal{U}\ne\emptyset$. In fact, it implies the existence of a probability measure $P$ on $\mathcal{A}$ such that $P\sim P^*$ and $(\pi_1,\ldots,\pi_n)$ is a $P$-martingale. Assuming $P^*$ has compact support, a characterization of $\mathcal{U}\ne\emptyset$ is provided by \cite[Theo. 3]{BPR2013}. According to the latter, $\mathcal{U}\ne\emptyset$ if and only if
$$P^*(f>a)>0\quad\text{ for all }a<0\text{ and all }f\text{ in the linear span of }H.$$

\end{document}